\documentclass[12pt,reqno]{article}
\UseRawInputEncoding

\usepackage[top=2.5cm,bottom=2.5cm,left=2.5cm,right=2.5cm]{geometry}
\usepackage{graphicx}
\usepackage{latexsym}
\usepackage[top=2.5cm,bottom=2.5cm,left=2.5cm,right=2.5cm]{geometry}
\usepackage{graphicx}
\usepackage{latexsym}
\usepackage{amsmath,amssymb}
\usepackage{amscd}

\usepackage[arrow,matrix]{xy}
\usepackage{graphicx}
\usepackage{xcolor}
\usepackage{times}
\usepackage[]{subfigure}
 \usepackage{cite}
 \usepackage{float}

\usepackage{hyperref}

\usepackage[colorinlistoftodos]{todonotes}
\newcommand\on{\operatorname}

\newcommand\Ric{\on{Ric}}

\usepackage{amscd,amsmath,amsthm,amssymb}

\theoremstyle{plain}
\newtheorem{theorem}{Theorem}[section]
\newtheorem{proposition}[theorem]{Proposition}
\newtheorem{corollary}[theorem]{Corollary}

\theoremstyle{definition}
\newtheorem{remark}[theorem]{Remark}
\newtheorem{definition}[theorem]{Definition}

\begin{document}

\title{ Ricci Solitons on the Poincar\'e upper half plane }
\author {{Abdou Bousso$^{1}$}\thanks{{
 E--mail: \texttt{abdoukskbousso@gmail.com} (A. Bousso)}},\texttt{ }Ameth  Ndiaye$^{2}$\footnote{{
 E--mail: \texttt{ameth1.ndiaye@ucad.edu.sn} (A. Ndiaye)}}\\
\begin{small}{$^{1}$D\'epartement de Math\'ematiques et Informatique, FST, Universit\'e Cheikh Anta Diop,  \.Dakar, S\'en\'egal.}\end{small}\\ 
\begin{small}{$^{2}$D\'epartement de Math\'ematiques, FASTEF, Universit\'e Cheikh Anta Diop, Dakar, Senegal.}\end{small}}
\date{}
\maketitle%


\begin{abstract} 
 In this paper, we characterize the Ricci soliton equations on the Poincar\'e upper half plane . First we classify all Ricci soliton and Ricci Bourguignon soliton in the half plane of Poincar\'e and after we generalize those equations in $\mathbb{H}^n$. We obtain some nice properties of the soliton about their geodesic flows.
\end{abstract}
\begin{small} {\textbf{MSC:}  53C20, 53C21.}
\end{small}\\
\begin{small} {\textbf{Keywords:} Ricci soliton, Ricci Bourguignon soliton, Hyperbolic plane, Riemannian metric.} 
\end{small}\\
\maketitle

\section{Introduction}

The Poincaré half-plane model is a method of describing the hyperbolic plane in non-Euclidean geometry by employing points in the well-known Euclidean plane. Each point in the hyperbolic plane is represented by a Euclidean point with coordinates \((x, y)\) whose \(y\) coordinate is greater than zero. The Poincaré upper half-plane is the generalization of the half-plane in dimension \(n \geq 2\).

In this work, we study the Ricci soliton one these manifolds. Einstein metrics play a significant role in the intersection of physics and mathematics. Numerous scholars have concentrated on this unique class of metrics and certain related structures known as Einstein type metrics during the past few decades \cite{Besse}. The idea of Ricci solitons is one of the most well-known of these \cite{Cao, Fernandez, Munteanu, Petersen1, Petersen2, Petersen3}, having been first proposed by Hamilton \cite{Hamilton} who introduced in 1982, the notion of Ricci flow to find a canonical metric on a smooth manifold.
 The Ricci flow is an evolution equation for metrics on a Riemannian manifold defined as follows:
 \begin{eqnarray}\label{flow}
     \frac{\partial}{\partial_t}g_{ij}=-2R_{ij}.
 \end{eqnarray}
 Ricci solitons are special solutions of the Ricci flow equation (\ref{flow}) of the form $g_{ij}=\mu(t)\phi^*_tg_{ij}$ with the initial
 condition $g_{ij}(0)=g_{ij}$, where $\phi_t$ are diffeomorphisms of the manifold and $\mu(t)$ is the scaling function.  A Ricci soliton is a generalization of an Einstein metric. A Ricci soliton is a Riemannian metric $g$ together with a vector field $X$ and $\lambda$ a scalar which satisfies
\begin{equation} \label{eq1}
\Ric+\frac{1}{2}\mathcal{L}_{X}g=\lambda g.
\end{equation}
In \cite{Manolo}, the author discuss some classification results for Ricci solitons, that is, self similar
 solutions of the Ricci Flow.
In detail, they took the equation point of view, trying to avoid the tools provided by
 considering the dynamic properties of the Ricci flow.\\
 The paper is organized as follow: in Section 2, we recall some basic notions which we use in this study. In Section 3, there are two subsections, first we study the Ricci soliton in the Poincar\'e half plane $\mathbb{H}^2$ with some properties of the flow of the soliton. And after that we generalized these result in height dimensions.

\section{Preliminaries}

In this section we recall all basic notions which we use for our main results. 
\subsection{Ricci solitons}
 Let  $(M^n, g)$ be an  $n$-dimensional Riemannian manifold, then we defined on $M^n$ the {\it Ricci-Bourguignon solitons} as a  self-similar solutions to {\it Ricci-Bourguignon flow} \cite{bib13*} defined by:
\begin{equation} \label{eq1}
\frac{\partial }{\partial t}g(t)=-2(\Ric-\rho S g), 
\end{equation}
where $S$ is the scalar curvature of the Riemannian metric $g$, $\Ric$ is the Ricci curvature tensor of the metric,  and $\rho$ is a real constant. 
 When  $\rho=0$ in ~\eqref{eq1}, then  we get a Ricci flow. 
Remark that for some  values of $\rho$ in equation~\eqref{eq1}, $\Ric-\rho Sg$ evolve into  the following situations:  \\

\textbf (1)  If  $\rho=\frac{1}{2},$ then it is an Einstein tensor $\displaystyle\Ric-\frac{S}{2}g$, 

\textbf (2)   If  $\rho=\frac{1}{n},$ then it is a traceless Ricci tensor $\displaystyle\Ric-\frac{S}{n}g$.

\begin{definition}\cite{bib2}
Let  $(M^n, g)$  be a  Riemannian manifold of dimension $n \geq 3$. Then it  is called  Ricci-Bourguignon soliton if
\begin{equation} \label{eq2}
\Ric+\frac{1}{2}\mathcal{L}_{X}g=(\lambda+\rho S)g,
\end{equation}
where  $\mathcal{L}_{X}$ denotes the Lie derivative operator along the vector field $X$ which is called soliton or potential, $\rho$  and  $\lambda$ are real constants and the soliton structure is  denoted by  $(M, g, X, \lambda, \rho)$. 
\end{definition}

If  $\lambda<0,$ $\lambda=0$ or  $\lambda>0$,  then the Ricci-Bourguignon soliton  is called   expanding, steady or  shrinking  respectively.


\begin{definition} \cite{Haj}
Let $(M^n, g)$ be a Riemannian manifold. Then it is called  $G$-Ricci-Bourguignon soliton  if there exist a vector field $X$, a constant $\lambda$ and a smooth function $G\neq 0$  such that 
\begin{equation} \label{eq3h}
\Ric+\frac{G}{2}\mathcal{L}_{X}g=(\lambda+\rho S)g.
\end{equation}
Hence it is  denoted by $(M^n, g, X, G, \lambda, \rho)$.
\end{definition}
Therefore, if $X=\nabla F$, we get a gradient $G$-Ricci-Bourguignon soliton.  Hence,  equation \eqref{eq3h} is rewritten as
\begin{equation} \label{eq4h}
\Ric+G\nabla^2 F=(\lambda+\rho S)g,
\end{equation}
where $\nabla^2 F$ is the Hessian of $F$.
Thus for $\rho=\frac{1}{2}$, we have a  gradient $G$-Einstein soliton  and for $\rho=\frac{1}{n}$, we get a gradient  $G$-almost traceless Ricci soliton. 
\begin{remark}\label{R0}
Let \(X\in\mathfrak{X}(\mathbb{H}^n)\) be a vector field. The triple \((\mathbb{H}^n,ds^2,X)\) is a Ricci soliton iff it is Ricci Bourguignon soliton. Indeed we have
\begin{itemize}
                \item Suppose that \((\mathbb{H}^n,ds^2,X,\lambda)\) is a Ricci soliton. You just have to choose $\lambda=\lambda'-n(n-1)\rho$ in the equation \eqref{eq3h} and we get the equation \eqref{eq4h}.
                \item Suppose that\((\mathbb{H}^n,ds^2,X,\lambda,\rho)\) is a Ricci  Bourguignon soliton. Just replace \(\lambda\) in the equation \eqref{eq4h}  by \(\lambda'+n(n-1)\rho\)  for getting the equation \eqref{eq3h}.
            \end{itemize}
\end{remark}
\begin{remark}
    Let \(X\in\mathfrak{X}(\mathbb{H}^2)\) be a vector field. If \((\mathbb{H}^2,ds^2,X,\lambda,\rho)\) is a Ricci Bourguignon soliton  or  \((\mathbb{H}^2,ds^2,X,\lambda)\) is Ricci soliton then \( X\) is conform. Indeed \(\operatorname{Ric}\) is proportional to \(ds^2\).
\end{remark}

\subsection{Geometric of the Poincar\'e upper half plane}
The Poincar\'e upper half plane is defining by

\begin{eqnarray}\label{Hn}
\mathbb{H}^n = \{(x_1, x_2, \dots, x_{n-1}, x_n) \in \mathbb{R}^n \mid x_n > 0\},
\end{eqnarray}
where \(x_n > 0\) is the coordinate associated with the "hyperbolic height". The associate metric of $\mathbb{H}^n$ is given by:
\begin{eqnarray}
ds^2 = \frac{dx_1^2 + dx_2^2 + \cdots + dx_{n-1}^2 + dx_n^2}{x_n^2}.
\end{eqnarray}
In tensor terms, the metric can be expressed as:
$g_{ij} = \frac{\delta_{ij}}{x_n^2}$,
where \(\delta_{ij}\) is the Kronecker symbol.\\
The geometry of the Poincaré upper half-plane is characterized by a constant negative curvature.
The following proposition allows us to get the Riici curvature of $\mathbb{H}^n$.
\begin{proposition}\label{Arxiv}\cite{Oneil}
    A Riemannian manifold $(M^n, g)$ has constant curvature $k$ if and
 only if $Ric = (n-1)kg$.
\end{proposition}
Since the curvature of the hyperbolic space $(\mathbb{H}^n, ds^2)$ is equal to $-1$, we have, by using the Proposition \ref{Arxiv}, that
 the  Ricci curvature is given by:
\begin{eqnarray}\label{Ric}
     Ric_{ij} = -(n-1)g_{ij},
\end{eqnarray}
where  \(g_{ij}\) is the metric tensor, 
and the scalar curvature \(S\) is given by:
\begin{eqnarray}\label{scal}
    S=Tr(Ric_{ij}) = -n(n-1).
\end{eqnarray}
Now we recall the formula of the Lie derivative.
\begin{definition}
    Let \((M^n,g)\) be a Riemannian manifold and let \(\mathfrak{X}(M)\) be the space of the smooth vector fields. In local coordinates, the Lie derivative of a metric $g$ in the direction of a vector field \(X=\sum\limits_{k=1}^nX_k(x_1,\cdots,x_n)\partial x_k\) is given by: \begin{equation}\label{f1}
        \left(\mathcal{L}_Xg\right)_{ij}=\sum_{k=1}^n\left(X_k(x_1,\cdots,x_n)\partial_kg_{ij}+g_{kj}\partial_iX_k(x_1,\cdots,x_n)+g_{ik}\partial_jX_k(x_1,\cdots,x_n)\right).
        \end{equation}
\end{definition}
Let \(X_{(x_1,\cdots,x_n)}=\sum\limits_{k=1}^nX_k(x_1,\cdots,x_n)\partial x_k\in \mathcal{X}(\mathbb{H}^n)\) be a vector field of $\mathbb{H}^n$. By (\ref{f1}), we have 
   \begin{equation}\label{f2} \begin{cases}
        \left(\mathcal{L}_{X(x_1,\cdots,x_n)}ds^2\right)_{ii}=\frac{2}{x_n^2}\left(\partial_iX_i(x_1,\cdots,x_n)-\frac{X_n(x_1,\cdots,x_n)}{x_n}\right)\\
\left(\mathcal{L}_{X(x_1,\cdots,x_n)}ds^2\right)_{ij}=\left(\mathcal{L}_{X(x_1,\cdots,x_n)}ds^2\right)_{ji}=\frac{\partial_jX_i(x_1,\cdots,x_n)+\partial_iX_j(x_1,\cdots,x_n)}{x_n^2}.
    \end{cases}
\end{equation}

Let \( F : \mathbb{H}^n \to \mathbb{R} \) be a smooth function. We want to calculate the components of the Hessian tensor \( \nabla^2(F) \) in the coordinates \( (x_1, \dots, x_n) \).

Let's remember the local formula:
\begin{eqnarray}\label{Hessian}
\nabla^2(F)\left(\frac{\partial}{\partial x^i}, \frac{\partial}{\partial x^j}\right) = \frac{\partial^2 F}{\partial x^i \partial x^j} - \Gamma^k_{ij} \frac{\partial F}{\partial x^k}, 
\end{eqnarray}
where the Christoffel symbols $\Gamma^k_{ij}$ are given the formula
\begin{eqnarray}\label{gamma}
\Gamma^k_{ij} = \frac{1}{2} g^{kl} \left( \frac{\partial g_{jl}}{\partial x^i} + \frac{\partial g_{il}}{\partial x^j} - \frac{\partial g_{ij}}{\partial x^l} \right).
\end{eqnarray}


The useful partial derivatives are:
\begin{itemize}
    \item if \( i \neq n \), \( \frac{\partial g_{ij}}{\partial x^n} = -\frac{2\delta_{ij}}{x_n^3} \)
    \item \( \frac{\partial g_{nn}}{\partial x^n} = -\frac{2}{x_n^3} \).
\end{itemize}
And then, the non-zero Christoffel symbols are:
\begin{align}\label{chris}
\Gamma^n_{ij} &= \frac{\delta_{ij}}{x_n}, \quad \text{si } i, j < n \nonumber\\
\Gamma^k_{nj} &= -\frac{\delta_{jk}}{x_n}, \quad \text{si } j < n\nonumber \\
\Gamma^n_{nn} &= -\frac{1}{x_n}.
\end{align}
Using the equations \eqref{Hessian}, \eqref{gamma} and \eqref{chris} we have \\
\textbf{Case 1 :} if \( i, j < n \), we have 
\[
\nabla^2(F)\left(\frac{\partial}{\partial x_i}, \frac{\partial}{\partial x_j}\right) 
= \frac{\partial^2 F}{\partial x_i \partial x_j} - \Gamma^n_{ij} \frac{\partial F}{\partial x_n}
= \frac{\partial^2 F}{\partial x_i \partial x_j} - \frac{\delta_{ij}}{x_n} \frac{\partial F}{\partial x_n}
\]
\textbf{Case 2 :} if  \( i = n \), \( j < n \) (and symmetric), then we have 
\[
\nabla^2(F)\left(\frac{\partial}{\partial x_n}, \frac{\partial}{\partial x_j}\right)
= \frac{\partial^2 F}{\partial x_n \partial x_j} - \sum_{k} \Gamma^k_{nj} \frac{\partial F}{\partial x_k}
= \frac{\partial^2 F}{\partial x_n \partial x_j} + \frac{1}{x_n} \frac{\partial F}{\partial x_j}
\]
\textbf{Case 3 :} if \( i = j = n \), then we have
\[
\nabla^2(F)\left(\frac{\partial}{\partial x_n}, \frac{\partial}{\partial x_n}\right)
= \frac{\partial^2 F}{\partial (x_n)^2} - \Gamma^n_{nn} \frac{\partial F}{\partial x_n}
= \frac{\partial^2 F}{\partial (x_n)^2} + \frac{1}{x_n} \frac{\partial F}{\partial x_n}.
\]
The Hessian matrix, denoted by $\nabla^2(F)$, is an $n\times n$ matrix whose elements are given by
$\nabla^2(F)_{ij} = \nabla^2(F)\left(\frac{\partial}{\partial x^i}, \frac{\partial}{\partial x^j}\right)$.
Using the three cases above we get

$$
\nabla^2(F) = \begin{pmatrix}
\frac{\partial^2 F}{\partial x_1^2} - \frac{1}{x_n} \frac{\partial F}{\partial x_n} & \frac{\partial^2 F}{\partial x_1 \partial x_2} & \cdots & \frac{\partial^2 F}{\partial x_1 \partial x_{n-1}} & \frac{\partial^2 F}{\partial x_1 \partial x_n} + \frac{1}{x_n} \frac{\partial F}{\partial x_1} \\
\frac{\partial^2 F}{\partial x_2 \partial x_1} & \frac{\partial^2 F}{\partial x_2^2} - \frac{1}{x_n} \frac{\partial F}{\partial x_n} & \cdots & \frac{\partial^2 F}{\partial x_2 \partial x_{n-1}} & \frac{\partial^2 F}{\partial x_2 \partial x_n} + \frac{1}{x_n} \frac{\partial F}{\partial x_2} \\
\vdots & \vdots & \ddots & \vdots & \vdots \\
\frac{\partial^2 F}{\partial x_{n-1} \partial x_1} & \frac{\partial^2 F}{\partial x_{n-1} \partial x_2} & \cdots & \frac{\partial^2 F}{\partial x_{n-1}^2} - \frac{1}{x_n} \frac{\partial F}{\partial x_n} & \frac{\partial^2 F}{\partial x_{n-1} \partial x_n} + \frac{1}{x_n} \frac{\partial F}{\partial x_{n-1}} \\
\frac{\partial^2 F}{\partial x_n \partial x_1} + \frac{1}{x_n} \frac{\partial F}{\partial x_1} & \frac{\partial^2 F}{\partial x_n \partial x_2} + \frac{1}{x_n} \frac{\partial F}{\partial x_2} & \cdots & \frac{\partial^2 F}{\partial x_n \partial x_{n-1}} + \frac{1}{x_n} \frac{\partial F}{\partial x_{n-1}} & \frac{\partial^2 F}{\partial x_n^2} + \frac{1}{x_n} \frac{\partial F}{\partial x_n}
\end{pmatrix}.
$$

\section{Main results}
\subsection{Poincar\'e half plane of dimension 2}
 Let \(X\) be a smooth vector field on $\mathbb{H}^2$. We have the first following result.
\begin{theorem}\label{T3}
    The manifold \((\mathbb{H}^2,ds^2,X, \lambda, \rho)\) is a \(G\)-Ricci Bourguignon soliton iff \begin{align*}
        X_ {(x,y)}&=\left( \int f(x)dx+(\lambda-2\rho+1)\left(\int\left(\int\frac{dy}{yG(x,y)}\right)dx+\int\frac{dx}{G(x,y)}\right)+h(y)\right)\partial x\\
        &+\left(f(x)y+(\lambda-2\rho+1)y\int\frac{dy}{yG(x,y)}\right)\partial y,
    \end{align*}  and \[h'(y)+f'(x)y=(-\lambda+2\rho-1)\left(\frac{\partial \left(y\int \frac{dy}{yG(x,y)}\right)}{\partial x}+\frac{\partial\left(\int\left(\int \frac{dy}{yG(x,y)}\right) \right)dx }{\partial y}+\frac{\partial \left(\int \frac{dx}{G(x,y)}\right)}{\partial y}\right)\] where \(f\in\mathcal{C}^\infty(\mathbb{R})\) and \(h\in\mathcal{C}^\infty(\mathbb{R}_*^+)\). And moreover the components of \(X\) are harmonic.
\end{theorem}
\begin{proof}
Consider the vector field in $\mathbb{H}^2$ given by
\(X_{(x,y)}=M(x,y)\partial x+N(x,y)\partial y \) such that the triple \((\mathbb{H}^2,ds^2,X)\) be a $G$-Ricci Bourguignon Soliton.
By the equations \eqref{f2} and \eqref{eq3h},  we get:
 \[\begin{cases}
    G(x,y)\left(\frac{M_x(x,y)}{y^2}-\frac{N(x,y)}{y^3}\right)=\frac{\lambda-2\rho+1}{y^2}\quad (L_1)\\
    M_y(x,y)+N_x(x,y)=0 \quad\quad\quad\quad\quad\quad (L_2)\\
    G(x,y)\left(\frac{N_y(x,y)}{y^2}-\frac{N(x,y)}{y^2}\right)=\frac{\lambda-2\rho+1}{y^2}.\quad \,\,(L_3)
    
\end{cases}\]
The line \((L_3)\) of the system gives \(N_y(x,y)-\frac{1}{y}N(x,y)=\frac{\lambda-2\rho+1}{G(x,y)}\).
Let us first solve the homogeneous differential equation, i.e. the equation: \[N_y(x,y)-\frac{N(x,y)}{y}=0\Rightarrow \frac{N_y(x,y)}{N(x,y)}=\frac{1}{y}\Rightarrow \ln\left(|N(x,y)|\right)=\ln(y)+a(x)\Rightarrow |N(x,y)|=\exp(a(x))y.\] 
The homogeneous solution is \(N_0(x,y)=f(x)y\) where \(|f(x)|=\exp(a(x))\).
Suppose that \(N_1(x,y)=k(x,y)y\) is a particular solution then \[(N_1)_y(x,y)=k_y(x,y)y+k(x,y)\] 
Which is to say that \[k_y(x,y)y=\frac{\lambda-2\rho+1}{G(x,y)}\Rightarrow k(x,y)=(\lambda-2\rho+1)\int\frac{1}{yG(x,y)}dy \]
\[N(x,y)=f(x)y+(\lambda-2\rho+1)y\int\frac{dy}{yG(x,y)}.\]
\((L_1)\Rightarrow M_x(x,y)-f(x)-(\lambda-2\rho+1)\int\displaystyle \frac{dy}{yG(x,y)}=(\lambda-2\rho+1)\frac{1}{G(x,y)}\)
\[M(x,y)=\int f(x)dx+(\lambda-2\rho+1)\left(\int\left(\int\frac{dy}{yG(x,y)}\right)dx+\int \frac{dx}{G(x,y)}\right)+h(y)\]
\((L_2)\) imply that  \[h'(y)+f'(x)y=(-\lambda+2\rho -1)\left(\frac{\partial \left(y\int \frac{dy}{yG(x,y)}\right)}{\partial x}+\frac{\partial\left(\int\left(\int \frac{dy}{yG(x,y)} \right)dx\right)}{\partial y}+\frac{\partial \left(\int\frac{dx}{G(x,y)}\right)}{\partial y}\right).\] In addition, the lines $(L_1)$ and $(L_3)$ directly imply that \(M_x(x,y)=N_y(x,y)\) and  \((L_2)\) gives \( M_y(x,y)=-N_x(x,y)\). So we have \[\begin{cases}
                  M_{xx}(x,y)=N_{xy}(x,y)\\
                  M_{xy}(x,y)=N_{yy}(x,y)\\
                  M_{yy}(x,y)=-N_{xy}(x,y)\\
                   M_{xy}(x,y)=-N_{xx}(x,y)
              \end{cases}\Rightarrow\begin{cases}
                   M_{xx}(x,y)+ M_{yy}(x,y)=0\\
                   N_{xx}(x,y)+N_{yy}(x,y)=0
              \end{cases}\Rightarrow \Delta M(x,y)=\Delta N(xx,y)=0.\] In the end, $M$ and $N$ are harmonic.
\end{proof}
If the vector field is gradient of a function we have the following theorem
\begin{theorem}\label{T4}
Let \( X \in\mathfrak{X}(\mathbb{H}^2) \) be a smooth vector field. If there exists a function
\( F \in \mathcal{C}^\infty(\mathbb{H}^2) \) such that 
$
X = \nabla F
$
and  \( (\mathbb{H}^2, ds^2, \nabla F,\lambda,\rho) \) be a   gradient-Ricci-Bourguignon soliton, then we have 
\[
F(x,y) =\frac{\frac{a}{2}\left(x^2+y^2\right)+a_1x +b}{y}+c,
\] where \(a,b,a_1,c\in\mathbb{R}\).
Moreover if the polynome \(P(x,y)=\frac{a}{2}\left(x^2+y^2\right)+a_1x +b\) does not admit any zero in \(\mathbb{H}^2\) then 
\[
 G(x, y) =  \frac{(\lambda-2\rho +1)y}{\frac{a}{2}\left(x^2+y^2\right)+a_1x +b}.
    \].
\end{theorem}
\begin{proof}
By the equation \eqref{eq4h}, we have the following system

\[
\left\{
\begin{aligned}
G(x,y) \left(\frac{\partial^2 F}{\partial x^2}-\frac{1}{y}\frac{\partial F}{\partial y}\right) &= \frac{\lambda - 2\rho + 1}{y^2} \quad \text{($L_1$)} \\
\left( \frac{\partial^2 F}{\partial x \partial y}+\frac{1}{y}\frac{\partial F}{\partial x}\right) &= 0 \quad \text{($L_2$)} \\
\left( \frac{\partial^2 F}{\partial y \partial x}+\frac{1}{y}\frac{\partial F}{\partial x}\right)  &= 0 \quad \text{($L_3$)} \\
G(x,y) \left(\frac{\partial^2 F}{\partial y^2}+\frac{1}{y}\frac{\partial F}{\partial y}\right) &= \frac{\lambda - 2\rho + 1}{y^2} \quad \text{($L_4$)}
\end{aligned}
\right.
\] If we multiply the line  \((L_2)\) by $y$ we get 

\[
y \frac{\partial^2 F}{\partial x \partial y} + \frac{\partial F}{\partial x} = 0
\Rightarrow \frac{\partial}{\partial y} \left( y \frac{\partial F}{\partial x} \right) = 0
\]
By integration we obtain
\[
y \frac{\partial F}{\partial x} = A(x) \Rightarrow \frac{\partial F}{\partial x} = \frac{A(x)}{y}
\]
and then we have
\[
F(x, y) = \frac{C(x)}{y} + B(y).
\]
So we have 
\begin{align*}
\frac{\partial F}{\partial x} &= \frac{C'(x)}{y}, &
\frac{\partial^2 F}{\partial x^2} &= \frac{C''(x)}{y} \\
\frac{\partial F}{\partial y} &= -\frac{C(x)}{y^2} + B'(y), &
\frac{\partial^2 F}{\partial y^2} &= \frac{2C(x)}{y^3} + B''(y)
\end{align*}
The equation ($L_1$) becomes  :
\[
G(x, y) \left( \frac{C''(x)}{y} + \frac{C(x)}{y^3} - \frac{B'(y)}{y} \right)
= \frac{\lambda - 2\rho + 1}{y^2}
\]
\begin{equation}\label{cd1}
    \Rightarrow G(x, y) \left( y^2 C''(x) + C(x) - y^2 B'(y) \right) = y(\lambda - 2\rho + 1) .
\end{equation}
If we make a substitution in the line ($L_4$), we get
\[
G(x, y) \left( \frac{C(x)}{y^3} + B''(y) + \frac{B'(y)}{y} \right)
= \frac{\lambda - 2\rho + 1}{y^2}.
\]
\begin{equation}\label{cd2}
    \Rightarrow G(x, y) \left( C(x) + y^3 B''(y) + y^2 B'(y) \right)
= y(\lambda - 2\rho + 1).
\end{equation}
If we compare the equations \eqref{cd1} and \eqref{cd2}, we obtain
\[
y^2 C''(x) + C(x) - y^2 B'(y) = C(x) + y^3 B''(y) - y^2 B'(y)
\]
that is 
\begin{eqnarray}\label{comp}
    y^2 C''(x)  = y^3 B''(y)+2y^2B'(y).
\end{eqnarray}
The equation \eqref{comp} is true if \(C''(x)=a\), where $a$ is a constant. So the equation \eqref{comp} becomes  \(ay^2=y^3B''(y)+2y^2B'(y)\) and we have  \(yB''(y)+2B'(y)=a\). We have \(B_0'(y)=\frac{e_1}{y^2}\) is a particular solution ($e_1$ is a constant). Suppose that \(B_2'(y)=\frac{k(y)}{y^2}\) directly involving from \(\frac{k'(y)}{y}=a\), then we get \(k(y)=\frac{a}{2}y^2+e_2\). So we have  \[B'(y)=\frac{e_1+e_2}{y^2}+\frac{a}{2}\Rightarrow B(y)=\frac{a}{2}y-\frac{e_1+e_2}{y}+c\] \[F(x,y)=\frac{\frac{a}{2}x^2+a_1x+a_2}{y}+\frac{a}{2}y-\frac{e_1+e_2}{y},\]
where $a_1, a_2, e_2$ are some constants.\\
Finally we have \[F(x,y)=\frac{\frac{a}{2}x^2+a_1x}{y}+\frac{a}{2}y+\frac{b}{y}+c\] where \(b=a_2-e_1-e_2\). 
The equation \eqref{cd1} gives 
\[G(x,y)\left(\frac{a}{2}x^2+a_1x+b+\frac{a}{2}y^2\right)=(\lambda-2\rho+1)y.\]
We have 
\begin{enumerate}
    \item  \(\lambda=2\rho-1\) if and only if \(\frac{a}{2}x^2+a_1x-b+\frac{a}{2}y^2=0\) so we get \(F(x,y)=c\) and \(G\) is any function.
    \item  \(\lambda\neq 2\rho-1\) if and only if \(\frac{a}{2}x^2+a_1x+b+\frac{a}{2}y^2\ne0\) for all \((x,y)\in\mathbb{H}^2\) so we have \[G(x,y)=\frac{(\lambda-2\rho+1)y}{\frac{a}{2}x^2+a_1x+b+\frac{a}{2}y^2}.\]
\end{enumerate}
\end{proof}
In the Theorem \ref{T4}, if we choose a special function $G$, we get the following proposition
Using the theorems above we get easily the following corollaries.
\begin{corollary}\label{T5}
    Let \(X\) be a non constant smooth vector field in $\mathbb{H}^2$.  If  \((\mathbb{H}^2,ds^2,X,\lambda,\rho)\) is a Ricci-Bourguignon soliton then \(X\) a Killing vector field. Moreover \begin{equation}
        X_{(x,y)}=\left(\frac{a}{2}\left(x^2-y^2\right)+bx+c\right)\partial x+\left(axy+by\right)\partial y.
    \end{equation}
\end{corollary}
\begin{corollary}\label{So}
    Let  \(X\) be smooth non-constant vector field in $\mathbb{H}^2$. Then the manifold \((\mathbb{H}^2,ds^2;X, \lambda, \rho)\) is not a gradient Ricci-Bourguignon soliton.

\end{corollary}
\begin{corollary}\label{T1}
   For all non constant vector fields \(X\in\mathfrak{X}(\mathbb{H}^2)\), if  \((\mathbb{H}^2,ds^2,X,\lambda)\) is a Ricci-soliton then \(X\) is Killing. Moreover   \begin{equation}
        X_{(x,y)}=\left(\frac{a}{2}\left(x^2-y^2\right)+bx+c\right)\partial x+\left(axy+by\right)\partial y.
    \end{equation} where \(a, b,c\in \mathbb{R}\).
\end{corollary}
\begin{corollary}\label{T2}
    Let \(X\) be a smooth non-constant vector field in $\mathbb{H}^2$. Then
    \((\mathbb{H}^2,ds^2,X, \lambda)\)is not a gradient Ricci- soliton, then \(X\) is constant.
   \end{corollary}
\begin{corollary}\label{Ss}
    Let \(X\in\mathfrak{X}(\mathbb{H}^2)\) such that \((\mathbb{H}^2,ds^2,X,\lambda,\rho)\) is \(G\)-Ricci-Bourguignon soliton . Then \(X\)  Killing iff \(\lambda=2\rho-1\). Moreover \begin{equation}
        X_{(x,y)}=\left(\frac{a}{2}\left(x^2-y^2\right)+bx+c\right)\partial x+\left(axy+by\right)\partial y,
    \end{equation} where \(a, b,c\in \mathbb{R}\).
\end{corollary}

\subsection{Poincar\'e upper half plane $\mathbb{H}^n$}
Now we can generalized the results in $\mathbb{H}^n$. 
\begin{theorem}\label{th1}
Let \(X\in \mathfrak{X}(\mathbb{H}^n)\) be a vector field. 
The manifold  \((\mathbb{H}^n,ds^2,X, \lambda, \rho)\) is a $G$-Ricci-Bourguignon soliton  iff \begin{align*}
        X(x_1,\cdots,x_n)&=\sum_{k=1}^{n-1}\left(\int f(x_1,\cdots,x_{n-1})dx_k\right)\partial x_k\\
        &+(\lambda+(n-1)(1-n\rho)\sum_{k=1}^{n-1}\int\left(\int \frac{1}{x_nG(x_1,\cdots,x_n)}dx_n\right)\partial x_k\\
        &+(\lambda+(n-1)(1-n\rho))\sum_{k=1}^{n-1}\left(\int \frac{1}{G(x_1,\cdots,x_n)}dx_k\right)\partial x_k\\
        &+\sum_{k=1}^{n-1}h_k(x_1,\cdots,x_{k-1},\widehat{x}_k,x_{k+1},\cdots,x_n)\partial x_k\\
        &+\left(f(x_1,\cdots,x_{n-1})x_n
        +(\lambda+(n-1)(1-n\rho))x_n\int \frac{1}{x_nG(x_1,\cdots,x_n)}dx_n\right)\partial x_n
    \end{align*}
where \(f:\mathbb{R}^{n-1}\rightarrow \mathbb{R}\) is a smooth function, and each \(h_k(x_1,...,x_{k-1},\widehat{x}_k,x_{k+1},...,x_n)\) is a smooth function of the other variables \(x_j\) with \(j\ne k\) and they verify all of the following constraints:\begin{itemize}
            \item[i)] For all \(k\in \lbrace 1,...,n-1\rbrace\)\begin{align*}&\frac{\partial h_k(x_1,...,x_{k-1},\widehat{x}_k,x_{k+1},...,x_n)}{\partial x_n}=-x_n\frac{\partial f(x_1,...,x_{n-1})}{\partial x_k}+\\&(-\lambda+(n-1)(n\rho-1))\left(\frac{\partial \left(x_n\int \frac{dx_n}{x_nG(x_1,...,x_n)}\right)}{\partial x_k}+\frac{\partial\left(\int\left(\int \frac{dx_n}{x_nG(x_1,...,x_n)}\right)dx_k+\int\frac{dx_k}{G(x_1,...,x_n)}\right)}{\partial x_n}\right).\end{align*}
            \item[ii)] For all \(k,j\in\lbrace1,...,n-1\rbrace\) such that \(k\ne j\).
            \begin{align*}
                &\frac{\partial\left(\int f(x_1,...,x_{n-1})dx_k\right)}{\partial x_j}+\frac{\partial \left(\int f(x_1,...,x_{n-1})dx_j\right)}{\partial x_k}+\frac{\partial h_k(x_1,...,x_{k-1},\widehat{x}_k,x_{k+1},...,x_n)}{\partial x_j}\\
                &+\frac{\partial h_j(x_1,...,x_{j-1},\widehat{x}_j,x_{j
                +1},...,x_n)}{\partial x_k}\\
               & =(-\lambda+(n-1)(n\rho-1))\left(\frac{\partial\left(\int \frac{dx_k}{G(x_1,...,x_n)}+\int\left(\int  \frac{dx_n}{x_nG(x_1,...,x_n)}\right)dx_k\right)}{\partial x_j}\right)\\
               &+(-\lambda+(n-1)(n\rho-1))\left( \frac{\partial \left(\int \frac{dx_j}{G(x_1,...,x_n)}+\int\left(\int \frac{dx_n}{x_nG(x_1,...,x_n)}\right)dx_j\right)}{\partial x_k}\right).
            \end{align*}
    \end{itemize}
   
    \end{theorem}
    \begin{proof}
Let $ X=\sum\limits_{k=1}^nX_k(x_1,\cdots,x_n)\partial x_k\in \mathcal{X}(\mathbb{H}^n)$ be a vector field. The equations \ref{f2} and \eqref{eq3h}, allow us to obtain the system:  
\[\begin{cases}
    G(x_1,\cdots,x_n)(\mathcal{L}_Xg)_{kk}=G(x_1,\cdots,x_n)\left(\frac{2}{x_n^2}\partial_kX_k-\frac{2}{x_n^3} X_n\right)=2\left(\frac{\lambda+(n-1)(1-n\rho)}{x_n^2}\right)\quad(L_8)\\\\
    (\mathcal{L}_Xg)_{kj}=(\mathcal{L}_Xg)_{jk}= \frac{\partial_kX_j+\partial_jX_k}{x_n^2}=0.\quad (L_9)
\end{cases}\]
The line $(L_8)$ gives $(X_n)_{x_n}-\frac{1}{x_n}X_n=\frac{\lambda+(n-1)(1-n\rho)}{G(x_1,\cdots,x_n)}$.
The solution of the homogeneous equation is \((X_n)_{0}=f(x_1,\cdots,x_{n-1})x_n.\) 
Now let us putting  \((X_n)_1=k(x_1,\cdots,x_n)x_n\) as particular 
 solution of $(L_8)$, then we have  \[\frac{\partial k(x_1,\cdots,x_n)}{\partial x_n}x_n=\frac{\lambda+(n-1)(1-n\rho)}{G(x_1,\cdots,x_n)}\Rightarrow k(x_1,\cdots,x_n)=(\lambda+(n-1)(1-n\rho))\int\frac{dx_n}{x_nG(x_1,\cdots,x_n)}.\]
The general solution of $(L_8)$ is
\[X_n(x_1,\cdots,x_n)=f(x_1,\cdots,x_{n-1})x_n+(\lambda+(n-1)(1-n\rho))x_n\int\frac{dx_n}{x_nG(x_1,\cdots,x_n)}.\]
\begin{align*}(L_8)\Rightarrow (X_k)_{x_k}&=\frac{X_n}{x_n}+\frac{\lambda+(n-1)(1-n\rho)}{G(x_1,\cdots,x_n)}\\
&=f(x_1,\cdots,x_{n-1})+(\lambda+(n-1)(1-n\rho))\int \frac{dx_n}{x_nG(x_1,\cdots,x_n)}+\frac{\lambda+(n-1)(1-n\rho)}{G(x_1,\cdots,x_n)}.\end{align*}
\begin{align*}
X_k(x_1,\cdots,x_n)&=\int f(x_1,\cdots,x_{n-1})\mathrm{d}x_k\\
&+(\lambda+(n-1)(1-n\rho))\left(\int\frac{dx_k}{G(x_1,\cdots,x_n)}
+\int\left(\int \frac{dx_n}{x_nG(x_1,\cdots,x_n)}\right)dx_k\right)\\
&+h_k(x_1,\cdots,x_{k-1},\widehat{x}_k,x_{k+1},\cdots,x_n).
\end{align*}
Since \[\frac{\partial X_k(x_1,\cdots,x_n)}{\partial x_j}+\frac{\partial X_j(x_1,\cdots,x_n)}{\partial x_k}=0\]
\[f\in\mathcal{C}^{\infty}\left(\mathbb{R}^{n-1}\right)\quad \text{and} \quad h_k\in\mathcal{C}^{\infty}\left(\left\{(x_1,\cdots,x_{k-1},\widehat{x}_k,x_{k+1},\cdots,x_n)\right\}\subset\mathbb{R}^{n-2}\times \mathbb{R}^{*}_+\right).\] 
So we have
\begin{itemize}
    \item  for \(k,j\in\lbrace 1,...,n-1\rbrace\) and \(k\ne j\):
\begin{align*}&\frac{\partial \left(\int f(x_1,...,x_{n-1})dx_k\right) }{\partial x_j}
+\frac{\partial\left(\int f(x_1,...,x_{n-1})dx_j\right)}{\partial x_k}+\frac{\partial h_k(x_1,\cdots,x_{k-1},\widehat{x}_k,x_{k+1},\cdots,x_{n})}{\partial x_j}\\
&+\frac{\partial h_j(x_1,...,x_{j-1},\widehat{x}_j,x_{j+1},...,x_n)}{\partial x_k}\\
&=\left(\lambda+(n-1)(n\rho-1)\right)\left(\frac{\partial \left(\int\frac{dx_k}{G(x_1,...,x_n)}+\int\left(\int\frac{dx_n}{x_nG(x_1,...,x_n)}\right)dx_k\right)}{\partial x_j}\right)\\
&+( \lambda+(n-1)(n\rho-1)\left(\frac{\partial\left(\int\frac{dx_j}{G(x_1,...,x_n)}+\int\left(\int\frac{dx_n}{x_nG(x_1,...,x_n)}\right)dx_j\right)}{\partial x_k}\right),
\end{align*}
\item for \(j=n\) \begin{align*}&\frac{\partial h_k(x_1,...,x_{k-1},\widehat{x}_k,x_{k+1},...,x_n)}{\partial x_n}+x_n\frac{\partial f(x_1,...,x_{n-1})}{\partial x_k}\\
&=(-\lambda+(n-1)(n\rho-1)\left(\frac{\partial\left(x_n\int \frac{dx_n}{x_nG(x_1,...,x_n)}\right)}{\partial x_k}+\frac{\partial\left(\int\left(\int\frac{dx_n}{x_nG(x_1,...,x_n)}\right)dx_k+\int \frac{dx_k}{G(x_1,...,x_n)}\right)}{\partial x_n}\right).\end{align*}
\end{itemize}
\end{proof}
    \begin{theorem}\label{pp}
 If \( X \in \mathfrak{X}(\mathbb{H}^2) \) is a smooth vector field. If there exists a smooth function \( F \in \mathcal{C}^\infty(\mathbb{H}^n) \) such that
$X = \nabla F$
and  \( (\mathbb{H}^n, ds^2, \nabla F,\lambda,\rho) \) is a $G$-gradient-Ricci Bourguignon soliton, then we have 
\[F(x_1,...,x_n)=\frac{1}{x_n}\sum_{i=1}^{n-1}\left(\frac{a}{2}x_i^2+b_ix_i\right)+\frac{a}{2}x_n+\frac{c}{x_n}+e\] where \(a,b_i,c,e\in\mathbb{R}\).
Moreover if the polynome \(P(x_1,...,x_n)=\frac{a}{2}x_n^2+c+\sum\limits_{k=1}^{n-1}\left(\frac{a}{2}x_k^2+b_kx_k\right)\) does not admit any zero in \(\mathbb{H}^n\) then we have
\[
 G(x_1,...,x_n) = \frac{\left(\lambda+(n-1)(1-n\rho)\right)x_n}{\frac{a}{2}x_n^2+c+\sum\limits_{k=1}^{n-1}\left(\frac{a}{2}x_k^2+b_kx_k\right)}.
    \]  \end{theorem}
\begin{proof}
The equation $\nabla^2(F) = \frac{\lambda+(n-1)(1-n\rho)}{x_n^2} I_n$ means that 
\begin{itemize}
      \item For $1 \le i < n$, we have
      \begin{enumerate}
          \item   $\frac{\partial^2 F}{\partial x_i^2} - \frac{1}{x_n}  \frac{\partial F}{\partial x_n} =\frac{\lambda+(n-1)(1-n\rho)}{x_n^2G(x_1,...,x_n)}$
    \item $\frac{\partial^2 F}{\partial x_i \partial x_j} = 0.$ 
      \end{enumerate}
\item For $i= n$, we have $\frac{\partial^2 F}{\partial x_n^2} + \frac{1}{x_n} \frac{\partial F}{\partial x_n} = \frac{\lambda+(n-1)(1-n\rho)}{x_n^2G(x_1,...,x_n)}.$

    \item For $i < n$, we have $\frac{\partial^2 F}{\partial x_n \partial x_i} + \frac{1}{x_n} \frac{\partial F}{\partial x_i} = 0.$ 
\end{itemize}
By the equations  $\frac{\partial^2 F}{\partial x_i \partial x_j} = 0$ for $i \neq j$ and $i, j < n$, we deduce that $F$ is at the form 
$$F(x_1, \dots, x_n) = F_1(x_1, x_n) + F_2(x_2, x_n) + \cdots + F_{n-1}(x_{n-1}, x_n) + \phi(x_n).$$

By differentiation with respect to $x_i$ (for $i < n$), we obtain $\frac{\partial F}{\partial x_i} = \frac{\partial F_i}{\partial x_i}(x_i, x_n)$. The condition $\frac{\partial^2 F}{\partial x_i \partial x_n} + \frac{1}{x_n} \frac{\partial F}{\partial x_i} = 0$ becomes :
$$\frac{\partial^2 F_i}{\partial x_n \partial x_i} + \frac{1}{x_n} \frac{\partial F_i}{\partial x_i} = 0\Rightarrow \frac{\partial\left(x_n\frac{\partial F_i(x_i,x_n)}{\partial x_i}\right)}{\partial x_n}=0$$
so we have $\frac{\partial F_i}{\partial x_i} = \frac{A_i(x_i)}{x_n}$. By integrating about $x_i$, we obtain  $$F_i(x_i, x_n) = \frac{\int  A_i(x_i)dx_i}{x_n} + B_i(x_n)\Rightarrow F_i(x_i,x_n)=\frac{C_i(x_i)}{x_n}+B_i(x_n).$$ 
So we have  \[F(x_1,...,x_n)=\frac{1}{x_n}\sum_{i=1}^{n-1}C_i(x_i)+B(x_n)\] where \[C_i(x_i)=\int A_i(x_i)dx_i\quad \forall \quad i\in\{1,...,n-1\}\quad \text{and}\quad B(x_n)=\phi(x_n)+\sum_{i=1}^{n-1}B_i(x_n).\]
Now consider the equation $\frac{\partial^2 F}{\partial x_i^2} - \frac{1}{x_n} \frac{\partial F}{\partial x_n} = \frac{\rho'}{x_n^2G(x_1,...,x_n)}$ for $i < n$. By using the expression of  $F$, we have :
\begin{equation}\label{1e}
    \frac{1}{x_n}C_i''(x_i) - \frac{1}{x_n} \left( B'(x_n)-\frac{1}{x_n^2}\sum_{k=1}^{n-1}C_k(x_k)\right) =  \frac{\rho'}{x_n^2G(x_1,...,x_n)}.
\end{equation}
The equation \eqref{1e} shows that for all \(i,j<n\) such that \(i\ne j\),  \(C''_i(x_i)=C''_j(x_j)\). This shows that \(C_i(x_i)=\frac{a}{2}x_i^2+b_ix_i+c_i\)
The quation $\frac{\partial^2 F}{\partial x_n^2} + \frac{1}{x_n} \frac{\partial F}{\partial x_n} = \frac{\lambda+(n-1)(1-n\rho)}{x_n^2G(x_1,...,x_n)}$ implies that \begin{equation*}
    B''(x_n)+\frac{2}{x_n^3}\sum_{k=1}^{n-1}\left(\frac{a}{2}x_k^2+b_kx_k+c_k\right)+\frac{1}{x_n}\left(B'(x_n)-\frac{1}{x_n^2}\sum_{k=1}^{n-1}\left(\frac{a}{2}x_k^2+b_kx_k+c_k\right)\right)=\frac{\lambda+(n-1)(1-n\rho)}{x_n^2G(x_1,...,x_n)}.
\end{equation*}
That is 
\begin{equation}\label{2e}
    B''(x_n)+\frac{1}{x_n^3}\sum_{k=1}^{n-1}\left(\frac{a}{2}x_k^2+b_kx_k+c_k\right)+\frac{1}{x_n}B'(x_n)=\frac{\lambda+(n-1)(1-n\rho)}{x_n^2G(x_1,...,x_n)}.
\end{equation}
So the equations \eqref{1e} and \eqref{2e} give \[\frac{a}{x_n} - \frac{1}{x_n} \left( B'(x_n)-\frac{1}{x_n^2}\sum_{k=1}^{n-1}\left(\frac{a}{2}x_k^2+b_kx_k+c_k\right)\right) =  B''(x_n)+\frac{1}{x_n^3}\sum_{k=1}^{n-1}\left(\frac{a}{2}x_k^2+b_kx_k+c_k\right)+\frac{1}{x_n}B'(x_n) \]
\[B''(x_n)+\frac{2}{x_n}B'(x_n)=\frac{a}{x_n}.\]
If we put \(\varphi(x_n)=B'(x_n)\) then we have \begin{eqnarray}\label{phi}\varphi'(x_n)+\frac{2}{x_n}\varphi(x_n)=\frac{a}{x_n},
\end{eqnarray}
where a homogeneous solution is \(\varphi_0(x_n)=\frac{e_1}{x^2_n}\). Let \(\varphi_2(x_n)=\frac{k(x_n)}{x_n^2}\) be a particular solution of the equation \eqref{phi}, then  \[\frac{k'(x_n)}{x_n^2}=\frac{a}{x_n}\Rightarrow k(x_n)=\frac{a}{2}x_n^2+e_2\] that is \(\varphi(x_n)=\frac{a}{2}+\frac{b}{x_n^2}\) where \(b=e_1+e_2\), which means that\[B(x_n)=\frac{a}{2}x_n-\frac{b}{x_n}+e\] and 
\[F(x_1,...,x_n)=\frac{1}{x_n}\sum_{i=1}^{n-1}\left(\frac{a}{2}x_i^2+b_ix_i\right)+\frac{a}{2}x_n+\frac{c}{x_n}+e\] where \(c=-b+\sum\limits_{i=1}^{n-1}c_i.\)
The equation \eqref{1e} implies that $$\frac{a}{x_n}-\frac{\varphi(x_n)}{x_n}+\frac{1}{x_n^3}\sum_{k=1}^{n-1}\left(\frac{a}{2}x_k^2+b_kx_k+c_k\right)=\frac{\lambda+(n-1)(1-n\rho)}{x_n^2G(x_1,...,x_n)}$$
\[\frac{a}{2}x_n^2+c+\sum\limits_{k=1}^{n-1}\left(\frac{a}{2}x_k^2+b_kx_k\right)=\frac{\left(\lambda+(n-1)(1-n\rho)\right)x_n}{G(x_1,...,x_n)}.\]
Now we have 
\begin{enumerate}
    \item  \(\lambda=-(n-1)(1-n\rho)\) if and only if \(\frac{a}{2}x_n^2+c+\sum\limits_{k=1}^{n-1}\left(\frac{a}{2}x_k^2+b_kx_k\right)=0\) so \(F(x_1,...,x_n)=e\) and \(G\) is any function.
    \item  \(\lambda\ne-(n-1)(1-n\rho)\) if and only if \(\frac{a}{2}x_n^2+c+\sum\limits_{k=1}^{n-1}\left(\frac{a}{2}x_k^2+b_kx_k\right)\ne0\)  for all \((x_1,...,x_n)\in\mathbb{H}^n\) so \[G(x_1,...,x_n)=\frac{\left(\lambda+(n-1)(1-n\rho)\right)x_n}{\frac{a}{2}x_n^2+c+\sum\limits_{k=1}^{n-1}\left(\frac{a}{2}x_k^2+b_kx_k\right)}.\]
\end{enumerate}
\end{proof}
 \begin{theorem}  \label{th3} 
    Let \(X\in \mathfrak{X}(\mathbb{H}^n)\). If \((\mathbb{H}^n,ds^2,X,\lambda,\rho)\) is Ricci-Bourguignon soliton  or  \((\mathbb{H}^n,ds^2,X,\lambda)\)is a Ricci soliton, then \(X\) is Killing.
Moreover we have \begin{align*}
        &X(x_1,\cdots,x_n)=\sum_{k=1}^{n-1}\left(\frac{a_k}{2} \left( x_k^2-\sum_{j\in\{1,...,n\}\setminus\{k\}}x_j^2\right) + \left( \sum_{i=1, i \neq k}^{n-1} a_i x_i + b \right) x_k+c_k \right)\partial x_k\\
        &+\left(\left( \sum_{k=1}^{n-1} a_k x_k + b\right)x_n\right)\partial x_n
    \end{align*} where \(a_k,b,c_k\) are constant reals for \(k\in\lbrace 1,...,n-1\rbrace\).
    \end{theorem}
    \begin{proof}
     To find the necessary and sufficient condition on \(X\) for that \((\mathbb{H}^n,ds^2,X,\lambda,\rho)\) to be Ricci-Bourguignon soliton, just take  \(G(x_1,\cdots,x_n)=1\) in the proof of the Theorem \ref{th1}. This directly implies that
constraints become\[\frac{\partial h_k(x_1,...,x_{k-1},\widehat{x}_k,x_{k+1},..,x_n)}{\partial x_n}+x_n\frac{\partial f(x_1,...,x_{n-1})}{\partial x_k}=(-\lambda+(n-1)(n\rho-1)\frac{x_k}{x_n}\] which implies that \[\frac{\partial^2f(x_1,..,x_{n-1})}{\partial^2x_k}=\frac{(-\lambda+(n-1)(n\rho-1))}{x^2_n}\] which is possible that if $\lambda=(n-1)(n\rho-1)$, so the vector field \(X\) is Killing. Moreover \(\forall k\in\lbrace1,...,n-1\rbrace\), we have  \[\frac{\partial^2f(x_1,...,x_{n-1})}{\partial^2x_k}=0.\] 
 For all \(k,j\in \lbrace 1,...,n-1\rbrace\) such that \(j\ne k\) \begin{align*}
                &\frac{\partial h_k(x_1,...,x_{k-1},\widehat{x}_k,x_{k+1},...,x_n)}{\partial x_j}+\frac{\partial h_j(x_1,...,x_{j-1},\widehat{x}_j,x_{j+1},...,x_n)}{\partial x_k}\\
                &=-\left(\frac{\partial\left(\int f(x_1,..x_{n-1})dx_k\right)}{\partial x_j } +\frac{\partial \left(\int f(x_1,...,x_{n-1})dx_j\right)}{\partial x_k}\right)
            \end{align*}
        so we have \begin{align*}
                &\frac{\partial^2 h_k(x_1,...,x_{k-1},\widehat{x}_k,x_{k+1},...,x_n)}{\partial x_n\partial x_j}+\frac{\partial^2 h_j(x_1,...,x_{j-1},\widehat{x}_j,x_{j+1},...,x_n)}{\partial x_n\partial x_k}=0
            \end{align*}
            which is to say that \begin{align*}
                &\frac{\partial}{\partial x_j}\left(\frac{\partial h_k(x_1,...,x_{k-1},\widehat{x}_k,x_{k+1},...,x_n)}{\partial x_n}\right)+\frac{\partial}{\partial x_k}\left(\frac{\partial h_j(x_1,...,x_{j-1},\widehat{x}_j,x_{j+1},...,x_n)}{\partial x_n}\right)=0.
            \end{align*}
            Thus we have \[-2x_n\frac{\partial^2f(x_1,...,x_{n-1})}{\partial x_k\partial x_j}=0\Rightarrow \frac{\partial^2f(x_1,...,x_{n-1})}{\partial x_k\partial x_j}=0.\]
            \begin{itemize}
                \item   Consider the partial derivative of $f$ with respect to $x_j$:
    \[
    \psi_j(x_1, \ldots, x_{n-1}) = \frac{\partial f(x_1,\ldots,x_{n-1})}{\partial x_j}.
    \]

\item   The given condition implies that the partial derivative of $\psi_j$ with respect to any other variable $x_k$ (where $k \in {1, ..., n-1}$) is zero:
    \[
    \frac{\partial \psi_j}{\partial x_k} = \frac{\partial}{\partial x_k} \left( \frac{\partial f}{\partial x_j} \right) = \frac{\partial^2 f}{\partial x_k \partial x_j} = 0.
    \]
    This is true for any $k = {1, \ldots, n-1}$. This means that $\psi_j$ does not depend on $x_k$ for all $k \neq j$. Therefore, $\psi_j$ is a function that depends only on $x_j$, or even a constant with respect to all variables.

\item   Let $\psi_j(x_1, \ldots, x_{n-1}) = a_j$, where $a_j$ is a constant (which may depend on the index $j$). Then we obtain:
    \begin{eqnarray}\label{eqt4}
    \frac{\partial f(x_1,\ldots,x_{n-1})}{\partial x_j} = a_j.
\end{eqnarray}

\item  Let's integrate this equation (\ref{eqt4})  with respect to $x_j$, considering the other variables as constants, we have:
    \[
    f(x_1, \ldots, x_{n-1}) = \int a_j \, dx_j = a_j x_j + \eta_j(x_1, \ldots, x_{j-1}, \widehat{x}_j, x_{j+1}, \ldots, x_{n-1})
    \]
    where $\eta_j$ is an arbitrary function of other variable $n-2$.
\item   To satisfy this form for all $j = 1, \ldots, n-1$ simultaneously, the function $f$ must be a linear combination of the variables $x_1, \ldots, x_{n-1}$ plus a constant, i.e:
    \[
    f(x_1, \ldots, x_{n-1}) = a_1 x_1 + a_2 x_2 + \cdots + a_{n-1} x_{n-1} + b = \sum_{i=1}^{n-1} a_i x_i + b
    \]
    where $a_1, \ldots, a_{n-1}$ and $b$ constant reals. So we have \[\int f(x_1,...,x_{n-1})dx_k=\frac{1}{2} a_k x_k^2 + \left( \sum_{i=1, i \neq k}^{n-1} a_i x_i + b \right) x_k+ \xi_k(x_1, \dots, x_{k-1},\widehat{x}_k, x_{k+1}, \dots, x_{n-1})\] and \[h_k(x_1,...,x_{k-1},\widehat{x}_k,x_{k+1},...,x_n)= -\frac{a_kx^2_n}{2}+ \alpha_k(x_1,...,x_{k-1},\widehat{x}_k,x_{k+1},...,x_{n-1})\]
 \end{itemize}
            \begin{align*}
   &-\frac{\partial \alpha_k(x_1,...,x_{k-1},\widehat{x}_k,x_{k+1},...,x_{n-1})}{\partial x_j}
  -\frac{\partial \alpha_j(x_1,...,x_{j-1},\widehat{x}_j,x_{j+1},...,x_{n-1})}{\partial x_k}=  \\
&a_jx_k+\frac{\partial \xi_k(x_1, \dots, x_{k-1},\widehat{x}_k, x_{k+1}, \dots, x_{n-1})}{\partial x_j}+a_kx_j
+\frac{\partial \xi_j(x_1, \dots, x_{j-1},\widehat{x}_j, x_{j+1}, \dots, x_{n-1})}{\partial x_k}.
 \end{align*}
 This shows that  \begin{align*}-\frac{\partial \alpha_k(x_1,...,x_{k-1},\widehat{x}_k,x_{k+1},...,x_{n-1})}{\partial x_j}&=a_kx_j+\frac{\partial \xi_k(x_1, \dots, x_{k-1},\widehat{x}_k, x_{k+1}, \dots, x_{n-1})}{\partial x_j},
     \end{align*}
that is
     \begin{align*}\alpha_k(x_1,...,x_{k-1},\widehat{x}_k,x_{k+1},...,x_{n-1})&=-\frac{a_k}{2}x_j^2- \xi_k(x_1, \dots, x_{k-1},\widehat{x}_k, x_{k+1}, \dots, x_{n-1})+\zeta_{kj}(x_i) \,\text{with } i\neq k,j
     \end{align*}
where \(\zeta_{kj}\) is a function that depends on all variables except \(x_k\) and \(x_j\).
Thus we get
     \begin{align*}
         \alpha_k(x_1,...,x_{k-1},\widehat{x}_k,x_{k+1},...,x_{n-1}) &= -\frac{a_k}{2}x_j^2 - \xi_k(x_1, \dots,x_{k-1}, \widehat{x}_k,x_{k+1}, \dots, x_{n-1})
+ \zeta_{kj}(x_i)  \text{ with } i \neq k, j.     \end{align*}
For  $j = j_1$, we have
      \begin{align*}
    \alpha_k(x_1, \dots, \widehat{x}_k, \dots, x_{n-1}) 
    &= -\frac{a_k}{2} x_{j_1}^2 
    - \xi_k(x_1, \dots, x_{k-1}, \widehat{x}_k, x_{k+1}, \dots, x_{n-1}) \\
    &\quad + \zeta_{k j_1}(x_i), \quad i \neq k, j_1.
\end{align*}

For $j = j_2$, we have
\begin{align*}
&: \\
\alpha_k(x_1,...,\widehat{x}_k,...,x_{n-1}) &= -\frac{a_k}{2}x_{j_2}^2 - \xi_k(x_1, \dots,x_{k-1}, \widehat{x}_k,x_{k+1}, \dots, x_{n-1}) \\
&+ \zeta_{kj_2}(x_i), \, i\neq k, j_2.
     \end{align*}
By equalization we get
     \begin{align*}
         &: \\
-\frac{a_k}{2}x_{j_1}^2 &- \xi_k(x_1, \dots,x_{k-1}, \widehat{x}_k,x_{k+1}, \dots, x_{n-1}) + \zeta_{kj_1}(x_i \neq k, j_1) = -\frac{a_k}{2}x_{j_2}^2 \\
&- \xi_k(x_1, \dots,x_{k-1}, \widehat{x}_k,x_{k+1}, \dots, x_{n-1})+ \zeta_{kj_2}(x_i),\,i \neq k, j_2) \\
\zeta_{kj_1}(x_i)_{i \neq k, j_1} - \zeta_{kj_2}(x_i)_{i \neq k, j_2} &= \frac{a_k}{2}x_{j_1}^2 - \frac{a_k}{2}x_{j_2}^2 \\
\zeta_{kj_1}(x_i)_{i\neq k, j_1} - \zeta_{kj_2}(x_i)_{ i \neq k, j_2} &= \frac{a_k}{2}(x_{j_1}^2 - x_{j_2}^2) + \beta_k( \{x_i\}_{i \neq k,j_1,j_2} )- \beta_k( \{x_i\}_{i \neq k,j_1,j_2} ).
     \end{align*}
The expression of $\zeta_{kj_1}$ is given by
     \begin{align*}
 &: \\
\zeta_{kj_1}( \{x_i\}_{i \neq k, j_1} ) &= -\frac{a_k}{2}x_{j_2}^2 + \beta_k( \{x_i\}_{i \neq k,j_1,j_2} ).
\end{align*}
And $\alpha_k$ becomes  
\begin{align*}&: \\
\alpha_k(x_1,...,x_{k-1},\widehat{x}_k,x_{k+1},...,x_{n-1}) &= -\frac{a_k}{2}x_{j_1}^2 - \xi_k(x_1,...,x_{k-1},\widehat{x}_k,x_{k+1},...,x_{n-1}) -\frac{a_k}{2}x_{j_2}^2 \\
&+ \beta_k( \{x_i\}_{i \neq k,j_1,j_2} ) 
\end{align*}
where \(\beta_{k}\) is a function that depends on all variables except \(x_k\),  \(x_{j_1}\) and \(x_{j_2}\). 
    This process allows us to obtain:\[\alpha_k(x_1,...,x_{k-1},\widehat{x}_k,x_{k+1},...,x_{n-1}) = -\sum_{j\in \lbrace 1,... n-1\rbrace\setminus\{k\}}\frac{a_k}{2}x_{j}^2 - \xi_k(x_1,...,x_{k-1},\widehat{x}_k,x_{k+1},...,x_{n-1})+c_k. \]
     So we have  \begin{align*}
        &X(x_1,\cdots,x_n)=\sum_{k=1}^{n-1}\left(\frac{a_k}{2} \left( x_k^2-\sum_{j\in\{1,...,n\}\setminus\{k\}}x_j^2\right) + \left( \sum_{i=1, i \neq k}^{n-1} a_i x_i + b \right) x_k+c_k \right)\partial x_k\\
        &+\left(\left( \sum_{k=1}^{n-1} a_k x_k + b\right)x_n\right)\partial x_n.
    \end{align*}
\end{proof}

 \begin{remark}
        The Remark \ref{R0} ensures the case of Ricci Soliton.
    \end{remark}
    
    \begin{corollary}\label{th2}
    Let \(X\in\mathfrak{X}(\mathbb{H}^n)\) be a non-constant vector field.  Then \((\mathbb{H}^n,ds^2,X, \lambda, \rho)\) is not a gradient Ricci soliton and \((\mathbb{H}^n,ds^2,X, \lambda)\) is not a gradient Ricci-Bourguignon soliton. 
      
\end{corollary}
     \begin{corollary}\label{co1}
    Let \((\mathbb{H}^n,ds^2,X,\lambda,\rho)\) be a $G$-Ricci-Bourguignon soliton where \(X\) is non-constant. Then \(X\) is Killing iff \(\lambda=(1-n)(1-n\rho)\). Moreover we have  \begin{align*}
        &X(x_1,\cdots,x_n)=\sum_{k=1}^{n-1}\left(\frac{a_k}{2} \left( x_k^2-\sum_{j\in\{1,...,n\}\setminus\{k\}}x_j^2\right) + \left( \sum_{i=1, i \neq k}^{n-1} a_i x_i + b \right) x_k+c_k \right)\partial x_k\\
        &+\left(\left( \sum_{k=1}^{n-1} a_k x_k + b\right)x_n\right)\partial x_n,
    \end{align*} where \(a_k,b,c_k\) are real constants \(\forall \quad k\in\lbrace 1,..., n-1\rbrace\).
\end{corollary}
\begin{remark}
     Let \((\mathbb{H}^n,ds^2,X,\lambda,\rho)\) be a $G$-Ricci-Bourguignon soliton. If $n\geq 3$, then the components of $X$ are not harmonic.
\end{remark}

\end{document}